\documentclass[12pt]{amsart}

\usepackage{hyperref}
\usepackage{amsmath}
\usepackage{amsthm}
\usepackage{amsfonts}
\usepackage{amssymb}
\usepackage{graphicx}
\DeclareGraphicsExtensions{.eps}

\newtheorem{theorem}{Theorem}[section]
\newtheorem{utv*}{Proposition}
\newtheorem{hyp*}{Conjecture}
\newtheorem{lemma}[theorem]{Lemma}

\newtheorem{defin}{Definition}
\newtheorem{zamech}{Remark}
\newtheorem*{th*}{Theorem}

\newcommand{\av}[2]{\langle #1\rangle_{_{\scriptstyle #2}}}

\def\sli{\sum\limits}
\def\ili{\int\limits}

\def\R{\mathbb{R}}

\begin{document}

\title[Logarithmic bump conditions]
{Logarithmic bump conditions and the two-weight boundedness of  Calder\'on--Zygmund operators}
\author{David Cruz-Uribe, SFO}
\address{Department of Mathematics, Trinity College}
\email{David.CruzUribe@trincoll.edu}

\author{Alexander Reznikov}
\address{Department of Mathematics,  Michigan State University, East
Lansing, MI 48824, USA}
\email{reznikov@ymail.com}

\author{Alexander Volberg}
\address{Department of Mathematics, Michigan State University, East
Lansing, MI 48824, USA}
\email{volberg@math.msu.edu}
\urladdr{http://sashavolberg.wordpress.com}

\thanks{The first author is supported by the Stewart-Dorwart faculty
  development fund at Trinity College and by
  grant MTM2009-08934 from the Spanish Ministry of Science and
  Innovation; the third author  is supported  by the NSF under the grant  DMS-0758552. }


\makeatletter
\@namedef{subjclassname@2010}{
  \textup{2010} Mathematics Subject Classification}
\makeatother




%
%


\renewcommand{\abstractname}{Abstract}
\begin{abstract}
We prove that if a pair of weights $(u,v)$ satisfies a sharp $A_p$-bump condition in the scale of all log bumps or certain loglog bumps,
then Haar shifts map $L^p(v)$ into $L^p(u)$ with a constant quadratic in
the complexity of the shift.  This in turn implies the two weight boundedness
for all Calder\'on-Zygmund operators.  This gives a partial
answer to a long-standing conjecture.  We also give a partial
result for a related
conjecture for weak-type inequalities.  To prove our main results we
combine several different approaches to these problems; in particular
we use many of the ideas developed to prove the $A_2$ conjecture. As a
byproduct of our work we also disprove a conjecture by
Muckenhoupt and Wheeden  on weak-type inequalities for the Hilbert transform. This is closely related to the
recent counterexamples of Reguera, Scurry and Thiele.
\end{abstract}

\date{}
\maketitle

\section{Introduction}

In this paper we prove several partial results related to a pair of long-standing conjectures in the theory of two-weight norm
inequalities.   To state the conjectures and our results we recall a few facts about Orlicz
spaces;  see \cite[Chapter~5]{CU-M-P-book} for complete details.
Given a Young function $A$, the complementary function $\bar{A}$ is the Young function that
satisfies
\[ t \leq A^{-1}(t) \bar{A}^{-1}(t) \leq 2t, \qquad t>0. \]
We will say that a Young
function $\bar{A}$ satisfies the $B_{p'}$ condition, $1<p<\infty$, if for
some $c>0$,
\[ \int_c^\infty \frac{\bar{A}(t)}{t^{p'}}\frac{dt}{t} < \infty. \]

If $A$ and $\bar{A}$ are  doubling (i.e., if $A(2t)\leq
CA(t)$, and similarly for $\bar{A}$), then $\bar{A} \in B_p$ if and only
if
\[ \int_c^\infty \left(\frac{t^{p}}{A(t)}\right)^{p'-1}\frac{dt}{t} <
\infty. \]

\begin{zamech}
  As we will see with specific examples below, if $\bar{A}\in B_{p'}$,
  then $\bar{A}(t)\lesssim t^{p'}$ and $A(t) \gtrsim t^{p}$.
\end{zamech}

 Given $p$, $1<p<\infty$, let $A$ and $B$ be Young functions such that $\bar{A}\in
B_{p'}$ and $\bar{B}\in B_p$.  We say that the pair of weights $(u,v)$
satisfies an $A_p$ bump condition with respect to $A$ and $B$ if
\begin{equation} \label{eqn:ap-bump}
 \sup_Q \|u^{1/p}\|_{A,Q} \|v^{-1/p}\|_{B,Q}  < \infty,
\end{equation}
where the supremum is taken over all cubes $Q$ in $\R^d$, and the
Luxemburg norm
is defined by
\[ \|f\|_{A,Q} = \inf\bigg\{ \lambda > 0 : \frac{1}{|Q|}\int_Q
  A\big(|f(x)|/\lambda\big)\,dx \leq 1 \bigg\}. \]
If  \eqref{eqn:ap-bump} holds, then it is conjectured that
\begin{equation} \label{eqn:Lp}
T : L^p(v)\rightarrow L^p(u).
\end{equation}
Similarly, if the pair $(u,v)$ satisfies the weaker condition
\begin{equation} \label{eqn:ap-bump-weak}
 \sup_Q \|u^{1/p}\|_{A,Q} \|v^{-1/p}\|_{p',Q}  < \infty,
\end{equation}
then the conjecture is that
\begin{equation} \label{eqn:Lp-weak}
T : L^p(v)\rightarrow L^{p,\infty}(u).
\end{equation}

The conditions \eqref{eqn:ap-bump} and \eqref{eqn:ap-bump-weak} are
referred to as $A_p$ bump conditions because they may be thought of
as the classical two-weight $A_p$ condition with the localized $L^p$
and $L^{p'}$ norms ``bumped up'' in the scale of Orlicz spaces.  These
conditions have a long history.  They first appeared in connection with
estimates for integral operators related to the spectral theory of
Schr\"odinger operators: see Fefferman~\cite{CF} and
Chang--Wilson--Wolff~\cite{ChWW}. These papers demonstrate a very
close connection with uncertainty principles;  for this aspect also
see the very interesting paper of P\'erez and Wheeden~\cite{PeWh}.
The bump condition considered in \cite{ChWW,CF} was the
Fefferman--Phong condition that used ``power'' bumps:  i.e., Young
functions of the form $A(t)=t^{rp}$, $r>1$.  Power bumps were
independently introduced by Neugebauer~\cite{neugebauer83}.   Bump
conditions in full generality were introduced by
P\'erez~\cite{Pe94,Pe94JL,Pe}.

The conjectured strong and weak-type inequalities for singular
integrals have been studied extensively, but the full results have
proved elusive.  The strong-type conjecture is true for
operators of bounded complexity (e.g., the Hilbert transform, the
Riesz transforms and the Buerling--Ahlfors operators):
see~\cite{CU-Ma-Pe}.
Lerner~\cite{Le} proved that it holds for any Calder\'on-Zygmund
operator if $p>n$.  

Very recently, it was
proved for $p=2$ in any dimension and for any Calder\'on--Zygmund operator using Bellman function techniques:
see~\cite{NRV1}.

\begin{theorem}
\label{thm:conjecture-thmp2}
Given $p=2$, suppose the pair of weights $(u,v)$ satisfies
\eqref{eqn:ap-bump},
where $\bar{A}\in B_{2}$ and $\bar{B}\in B_{2}$.
Then every Calder\'on-Zygmund singular integral operator $T$ satisfies
$\|Tf\|_{L^2(u)}\leq C\|f\|_{L^2(v)}$, where $C$ depends only on $T$,
the dimension $d$, and the suprema in
\eqref{eqn:ap-bump}.
\end{theorem}

\begin{zamech}
  Extending Theorem~\ref{thm:conjecture-thmp2} to the case $p\neq 2$,
  and especially strengthening it by replacing the two-sided bump
  conditions \eqref{eqn:ap-bump} by the weaker, one-sided conditions
  \eqref{eqn:separated-bumps} and \eqref{eqn:separated-bumps-weak}
  discussed below, has proved to be very  difficult.
\end{zamech}

Certain additional results are known in the special case
that $A$ and $B$ are ``log-bumps'': that is, of the form
\begin{equation}
\label{eqn:left-bump}
A(t) = t^p\log(e+t)^{p-1+\delta}, \quad \bar{A}(t) \approx
\frac{t^{p'}}{\log(e+t)^{1+\delta'}} \,,
\end{equation}
\begin{equation}
 \label{eqn:right-bump}
B(t)=t^{p'}\log(e+t)^{p'-1+\delta}, \qquad \bar{B}(t) \approx
 \frac{t^{p}}{\log(e+t)^{1+\delta''}}\,,
 \end{equation}
where $\delta>0$, $\delta'=\delta/(p-1)$, $\delta''=\delta/(p'-1)$. But even in this case the result for all Calder\'on--Zygmund operators was unknown.
The weak-type conjecture is only known for log bumps:
see~\cite{cruz-uribe-perez99}.  For a complete history of both
conjectures and these partial results, we refer the reader to the
work of Cruz-Uribe, P\'erez, and Martell \cite{CU-Ma-Pe,CU-M-P-book,CU-Ma-Pe07,cruz-uribe-perez00b,CU-P-pisa}, and Treil, Volberg,
and Zheng~\cite{TVZ},
and the extensive references they contain.

\smallskip

One can motivate the conjectures \eqref{eqn:ap-bump} $\Rightarrow$
\eqref{eqn:Lp} and \eqref{eqn:ap-bump-weak} $\Rightarrow$
\eqref{eqn:Lp-weak} (and the related conjectures we  consider
below) by considering a pair of conjectures due to Muckenhoupt and Wheeden.
First, they conjectured that a singular integral operator (in
particular, the Hilbert transform) satisfies \eqref{eqn:Lp} provided
that the Hardy-Littlewood maximal operator satisfies
\begin{gather*}  M : L^p(v) \rightarrow L^p(u), \\ M : L^{p'}(u^{1-p'})
\rightarrow L^{p'}(v^{1-p'}).
\end{gather*}
They also conjectured that \eqref{eqn:Lp-weak} holds if the maximal
operator satisfies the second, $L^{p'}$ inequality.
P\'erez \cite{Pe} (see also \cite{CU-M-P-book}) proved that a
sufficient condition for each of these
estimates to hold for $M$ is that the pair $(u,v)$ satisfies
\begin{gather} \label{eqn:separated-bumps}
\sup_Q \|u^{1/p}\|_{p,Q} \|v^{-1/p}\|_{B,Q}  < \infty, \\
\label{eqn:separated-bumps-weak}
\sup_Q \|u^{1/p}\|_{A,Q} \|v^{-1/p}\|_{p',Q}  < \infty;
\end{gather}
in particular, both these conditions hold if \eqref{eqn:ap-bump}
holds.

Though intuitively appealing, both of the Muckenhoupt-Wheeden
conjectures are false.  A
counter-example to the strong-type conjecture was
recently found by Reguera and Scurry \cite{RS}.   The weak-type conjecture  is an easy consequence of the
two-weight, weak $(1,1)$ conjecture (also due to Muckenhoupt and
Wheeden), but this was recently proved false by Reguera and
Thiele~\cite{RT}.  While this does not show the conjecture false, it
strongly suggests that it is.  And as a byproduct of our
approach to our main results we show that the weak-type conjecture is
also false; as a consequence we get another proof that their weak
$(1,1)$ conjecture is false.

\smallskip

Given the falsity of the Muckenhoupt-Wheeden conjectures (even for
$p=2$), the $A_p$ bump conjectures become even more
interesting. Moreover, Theorem \ref{thm:conjecture-thmp2} and the
other results listed above strongly suggest that it should hold in the
full range of $p$, dimensions, and Calder\'on--Zygmund operators.
Here we consider two even stronger conjectures, motivated by the fact
that the ``separated'' bump conditions~\eqref{eqn:separated-bumps}
and~\eqref{eqn:separated-bumps-weak} are sufficient for the maximal
operator inequalities in the original conjecture.

\begin{hyp*} \label{c:conjecture-thm}
Given $p$, $1<p<\infty$, suppose the pair of weights $(u,v)$ satisfies
\eqref{eqn:separated-bumps} and \eqref{eqn:separated-bumps-weak},
where $\bar{A}\in B_{p'}$ and $\bar{B}\in B_{p}$.
Then every Calder\'on-Zygmund singular integral operator $T$ satisfies
$\|Tf\|_{L^p(u)}\leq C\|f\|_{L^p(v)}$, where $C$ depends only on $T$,
the dimension $d$, and the suprema in
\eqref{eqn:separated-bumps} and \eqref{eqn:separated-bumps-weak}.
\end{hyp*}

\begin{hyp*} \label{c:conjecture-weak-thm}
Given $p$, $1<p<\infty$, suppose the pair of weights $(u,v)$ satisfies
\eqref{eqn:separated-bumps-weak} where $\bar{A}\in B_{p'}$.
Then every Calder\'on-Zygmund singular integral operator $T$ satisfies
$\|Tf\|_{L^{p,\infty}(u)}\leq C\|f\|_{L^{p}(v)}$, where $C$ depends only on $T$,
the dimension $d$, and the supremum in \eqref{eqn:separated-bumps-weak}.
\end{hyp*}

\bigskip

We can prove  Conjecture~\ref{c:conjecture-thm} in the special case
when $A,\, B$ are log bumps.

\begin{theorem} \label{thm:conjecture-thm}
Given $p$, $1<p<\infty$, suppose the pair of weights $(u,v)$ satisfies
\eqref{eqn:separated-bumps} and \eqref{eqn:separated-bumps-weak},
where $A$ and $B$ are log bumps of the form \eqref{eqn:left-bump} and \eqref{eqn:right-bump}.
Then every Calder\'on-Zygmund singular integral operator $T$ satisfies
$\|Tf\|_{L^p(u)}\leq C\|f\|_{L^p(v)}$, where $C$ depends only on $T$,
the dimension $d$, and the suprema in
\eqref{eqn:separated-bumps} and \eqref{eqn:separated-bumps-weak}.
\end{theorem}

Our techniques also immediately yield
Conjecture~\ref{c:conjecture-weak-thm} for log bumps.  This gives a
new proof of the result originally proved
in~\cite{cruz-uribe-perez99}; for completeness we include it here.

\begin{theorem} \label{thm:conjecture-weak-thm}
Given $p$, $1<p<\infty$, suppose the pair of weights $(u,v)$ satisfies
\eqref{eqn:separated-bumps-weak} where $A$ is a log bump of the form \eqref{eqn:left-bump}.
Then every Calder\'on-Zygmund singular integral operator $T$ satisfies
$\|Tf\|_{L^{p,\infty}(u)}\leq C\|f\|_{L^{p}(v)}$, where $C$ depends only on $T$,
the dimension $d$, and the supremum in \eqref{eqn:separated-bumps-weak}.
\end{theorem}

\begin{zamech}
Theorems~\ref{thm:conjecture-thm} and~\ref{thm:conjecture-weak-thm}
are both sharp, in the sense that if we take $\delta=0$ in the
definition of $A$ or $B$, then there exist
pairs of weights that satisfy the bump conditions but such that the
corresponding norm inequalities are false.  For details,
see~\cite{CU-M-P-book}.
\end{zamech}

Our proof of Theorems~\ref{thm:conjecture-thm}
and~\ref{thm:conjecture-weak-thm} depends heavily on the machinery
developed to prove the one-weight $A_2$ conjecture \cite{CU-Ma-Pe,H, HPTV}.  In
turn many of the techniques used to prove the $A_2$ conjecture have
their genesis in nonhomogeneous Harmonic Analysis.  In particular, they
go back to the random geometric constructions introduced
in~\cite{NTV97, NTV98, NTV-acta}.  For a  summary of these results,
see~\cite{Vo}.

\medskip

Our method of proof can also be adapted to prove
Conjectures~\ref{c:conjecture-thm} and~\ref{c:conjecture-weak-thm} for
a subset of the class of bump functions referred to as loglog-bumps:
\begin{equation}
\label{llA}
A(t) = t^p\log(e+t)^{p-1}\log\log(e^e+t)^{p-1+\delta} \quad \bar{A}(t) \approx
\frac{t^{p'}}{\log(e+t)\log\log (e^e +t)^{1+\delta'}},
\end{equation}
\begin{equation}
 \label{llB}
B(t)=t^{p'}\log(e+t)^{p'-1}\log\log(e^e+t)^{p'-1+\delta}, \quad \bar{B}(t) \approx
 \frac{t^{p}}{\log(e+t)\log\log (e^e +t)^{1+\delta''}},
 \end{equation}
where $\delta>0$.  These bump conditions are well known, but have been
difficult to work with.  As was noted in~\cite[p.~107]{CU-M-P-book},
up until now, no results were known for loglog-bumps that were not
proved for bump conditions in general.  However, we can prove the
following results, both of which are new.

\begin{theorem} \label{thmllst}
Given $p$, $1<p<\infty$, suppose the pair of weights $(u,v)$
satisfy~\eqref{eqn:separated-bumps} and~\eqref{eqn:separated-bumps-weak}, where $A$ and $B$ are loglog-bumps of
the form \eqref{llA} and \eqref{llB} with {\bf $\delta$ sufficiently
large}.   Then every Calder\'on-Zygmund singular integral operator $T$
satisfies $\|Tf\|_{L^p(u)} \leq C\|f\|_{L^p(v)}$, 
 where $C$ depends only on $T$, the dimension $d$ and the suprema
in \eqref{eqn:separated-bumps} and~\eqref{eqn:separated-bumps-weak}.
\end{theorem}

\begin{theorem} \label{thmllw}
Given $p$, $1<p<\infty$, suppose the pair of weights $(u,v)$
satisfies~\eqref{eqn:separated-bumps-weak} where $A$ is a  loglog-bump of
the form \eqref{llA} with {\bf $\delta$ sufficiently
large}.   Then every Calder\'on-Zygmund singular integral operator $T$
satisfies $\|Tf\|_{L^{p,\infty}(u)} \leq C\|f\|_{L^p(v)}$, 
 where $C$ depends only on $T$, the dimension $d$ and the supremum
 in~\eqref{eqn:separated-bumps-weak}. 
\end{theorem}


\begin{zamech}
It will follow from the proof that it suffices to take
$\delta>(p'-1)^{-1}$ in~\eqref{llA} and $\delta>(p-1)^{-1}$
in~\eqref{llB}.  What is important is that we cannot take any
$\delta>0$ as conjectured.  This restriction should be compared to the
restriction $p>n$ in~\cite{CU-Ma-Pe07}, or the restriction that
$\delta>1$ given in for certain special classes of weights (the
so-called factored weights) in~\cite[Section~9.2]{CU-M-P-book}.  
\end{zamech}

The remainder of this paper is organized as follows.  In
Section~\ref{section:prelim} we reformulate our results and reduce the
problem to proving the corresponding results for a general class of
dyadic shift operators (Theorem~\ref{thm:dyadic-result}
and~\ref{thm:dyadic-result-weak}).  In Section~\ref{section:proof} we
prove Theorem~\ref{thm:dyadic-result}.  It is important to note that
in most of the proof we only need to assume that $\bar{A}\in B_{p'},
\bar{B}\in B_p$; only at one step are we forced to assume that $A,\,B$
are log bumps.  In Section~\ref{section:proof-weak} we describe the
(minor) changes required to prove
Theorem~\ref{thm:dyadic-result-weak}.   In Section ~\ref{ll}, we 
reformulate and prove Theorems~\ref{thmllst} and~\ref{thmllw}. 
Finally, in Section~\ref{ce} we
show that the Muckenhoupt-Wheeden conjecture for the weak-type
inequality is false.  

\bigskip

\noindent{\bf Acknowledgements.} The authors are very grateful to the
American Institute of Mathematics, where this work was essentially
done, for their hospitality.
The  authors also want to thank Michael Lacey for his insightful remarks.

\section{Preliminary Results}
\label{section:prelim}

Hereafter, we will use the notation
\[ \av{f}{Q}=\frac{1}{|Q|}\ili_Q f(x)dx. \]
We also restate our weighted norm inequalities in an equivalent form.
Let $\sigma=v^{1-p'}$; then we can rewrite \eqref{eqn:separated-bumps}
and \eqref{eqn:separated-bumps-weak}
as
\begin{gather}
\label{eqn:alt-separated-bumps}
\sup_Q \av{u}{Q}^{1/p} \|\sigma^{1/p'}\|_{B,Q}  < \infty, \\
\label{eqn:alt-separated-bumps-weak}
 \sup_Q \|u^{1/p}\|_{A,Q} \av{\sigma}{Q}^{1/p'}  < \infty.
\end{gather}
By the properties of the Luxemburg norm we have that either condition
implies the
two-weight $A_p$ condition:
\begin{equation} \label{eqn:alt-twowt-Ap}
 \sup_Q \av{u}{Q}^{1/p} \av{\sigma}{Q}^{1/p'}  < \infty.
\end{equation}
Similarly, we can restate the conclusions of
Theorems~\ref{thm:conjecture-thm} and~\ref{thm:conjecture-weak-thm}
as
\[ \|T(f\sigma)\|_{L^p(u)}\leq C\|f\|_{L^p(\sigma)}, \qquad \|T(f\sigma)\|_{L^{p,\infty}(u)}\leq C\|f\|_{L^p(\sigma)}. \]

\medskip

The $B_p$ condition is closely connected to a generalization of the
maximal operator.  Recall that the Hardy-Littlewood maximal operator
is defined to be
\[ Mf(x) = \sup_{Q\ni x} \av{|f|}{Q}= \sup_{Q\ni x} \|f\|_{1,Q}. \]
Given a Young function $A$, we define the Orlicz maximal operator
$M_A$ by
\[ M_A f(x):=\sup_{Q\ni x} \|f\|_{A,Q}. \]
The following result is due to P\'erez \cite{Pe} (see also
\cite{CU-M-P-book}).

\begin{theorem} \label{thm:perez-orlicz-max}
Fix $p$, $1<p<\infty$, and let $A$ be a Young function such that $A\in
B_p$.  Then $M_A : L^p \rightarrow L^p$.
\end{theorem}

The $B_p$ condition is also sufficient for a two-weight norm
inequality for the Hardy-Littlewood maximal operator.  This result is
also due to P\'erez \cite{Pe,CU-M-P-book}.

\begin{theorem} \label{thm:perez-twowt-max}
Fix $p$, $1<p<\infty$, and let $B$ be a Young function such that
$\bar{B}\in B_p$.  If the pair of weights $(u,\sigma)$ satisfies
\begin{equation} \label{eqn:max1}
 \sup_Q \av{u}{Q}^{1/p} \|\sigma^{1/p'}\|_{B,Q}  < \infty,
\end{equation}
then
\begin{equation} \label{eqn:max2}
 \|M(f\sigma)\|_{L^p(u)} \leq C\|f\|_{L^p(\sigma)}.
\end{equation}
\end{theorem}

\begin{zamech}
The bump condition \eqref{eqn:max1} is necessary in the following
sense:  suppose that $B$ is a function such that
whenever~\eqref{eqn:max1} holds, the maximal operator satisfies
\eqref{eqn:max2}.  Then $\bar{B}\in B_p$.  See~\cite{Pe}.
\end{zamech}

\bigskip

We now turn to the definition of the dyadic Haar shift operators that
will replace an arbitrary Calder\'on-Zygmund operator.

\begin{defin}
Given a dyadic cube $Q$, $h_Q$ is a (generalized) Haar function
associated to a cube $Q$ if
$$
h_Q(x)=\sli_{Q'\in ch(Q)} c_{Q'} \chi_{Q'}(x),
$$
where $ch(Q)$ is the set of dyadic children of $Q$ and $|c_{Q'}|\leq 1$.
\end{defin}

\begin{defin}
We say that an operator $S$ has a Haar shift kernel of complexity $(m,n)$ if
$$
Sf(x) = \sli_{Q} S_Q(f),
$$
where
$$
S_Q(f)=\frac{1}{|Q|}\sum_{\substack{Q', Q^{\prime\prime} \subset
    Q \\ \ell(Q')=2^{-n}\ell(Q) \\ \ell(Q^{\prime\prime})=2^{-m}\ell(Q)}} (f, h_{Q'})h_{Q^{\prime\prime}}
$$
and $h_{Q'}$ and $h_{Q^{\prime\prime}}$ are generalized Haar
functions associated to the cubes $Q'$ and $Q^{\prime\prime}$
respectively.  We say that $S$ is a Haar shift of complexity
$(m,n)$ if it has a Haar shift kernel of complexity $(m,n)$, and it is bounded on $L^2(dx)$.
\end{defin}

\medskip

By the decomposition theorem of Hyt\"onen~\cite{H,H2}, to prove
Theorems~\ref{thm:conjecture-thm} and~\ref{thm:conjecture-weak-thm} it will suffice to prove that they
hold for Haar shift operators of complexity $(m,n)$ with a constant that grows polynomially
in $\tau=\max(m,n)+1$.   More precisely we will prove the following.

\begin{theorem} \label{thm:dyadic-result} Given $p$, $1<p<\infty$,
  suppose $A$ and $B$ are log-bumps of the form \eqref{eqn:left-bump},
  \eqref{eqn:right-bump}, and  the pair of weights $(u,\sigma)$ satisfies
  \eqref{eqn:alt-separated-bumps} and
  \eqref{eqn:alt-separated-bumps-weak}.  Given any dyadic shift $S$ of
  complexity $(m,n)$, $\tau=\max(m,n)+1$, $\|S(f\sigma)\|_{L^p(u)}\leq
  C\tau^3\|f\|_{L^p(\sigma)}$, where $C$ depends only on the dimension
  $d$ and the suprema in \eqref{eqn:alt-separated-bumps} and
  \eqref{eqn:alt-separated-bumps-weak}.
\end{theorem}

\begin{theorem} \label{thm:dyadic-result-weak}
Given $p$, $1<p<\infty$, suppose    $A$ is a log-bump of the form
\eqref{eqn:left-bump},  and  the pair of weights $(u,\sigma)$ satisfies 
\eqref{eqn:alt-separated-bumps-weak}.  Given any dyadic shift $S$ of complexity $(m,n)$,
$\|S(f\sigma)\|_{L^{p,\infty}(u)}\leq C\tau^3\|f\|_{L^p(\sigma)}$, where $C$ depends only on
the dimension $d$ and the supremum in \eqref{eqn:alt-separated-bumps-weak}.
\end{theorem}

\section{Proof of Theorem~\ref{thm:dyadic-result}}
\label{section:proof}

To prove the strong-type inequality we follow the argument used by Hyt\"onen and Lacey~\cite{HyLa} in
the one-weight case, which in turn refines the proof
given in~\cite{CU-Ma-Pe}.   Fix a function $f$ that is bounded and has
compact support.  For each $N>0$, let $Q_N = [-2^N,2^N]^d$.   By
Fatou's lemma,
\[ \|S(f\sigma)\|_{L^p(u)} \leq \liminf_{N\rightarrow\infty}
\left(\int_{Q_N} |S(f\sigma)(x)-
  m_{S(f\sigma)}|^p u(x)\,dx\right)^{1/p}, \]
where $m_{S(f\sigma)}$ is the median value of
$S(f\sigma)$ on $Q_N$.  Fix $N$.  Using  the remarkable decomposition
theorem of Lerner ~\cite{Le}, they show that there exists a family of dyadic cubes
$\mathcal{L}=\{Q_j^k\}$ and pairwise disjoint sets $\{E_j^k\}$ such that
$E_j^k\subset Q_j^k$, $|E_j^k|\geq \frac{1}{2}|Q_j^k|$, and
\begin{multline} \label{eqn:strong-reduction}
\left(\int_{Q_N} |S(f\sigma)(x)-
  m_{S(f\sigma)}|^p u(x)\,dx\right)^{1/p}  \\\leq
C\tau \|M(f\sigma)\|_{L^p(u)}+ C\,\tau\sum_{i=1}^{\tau} \| \sli_{j,k}
\av{|f|\sigma}{(Q_j^k)^i}\chi_{Q_j^k} \|_{L^p(u)}.
\end{multline}
(Here, given a dyadic cube $Q$, $Q^i$ denotes the $i$-th parent of
$Q$.)  The linear dependence on $\tau$ in \eqref{eqn:strong-reduction}
can be found in \cite{HyLa}: see Lemma~2.4 and the discussion
following it.  Alternatively, we can deduce it from the argument in
\cite{CU-Ma-Pe} if: 1) we combine it with the unweighted weak-type
estimate with the right dependence on $\tau$ in \cite{H, HPTV};  2)
we precede the weak-type estimate of $1_Q S (1_{Q^{\tau}}f)$ by a
careful pointwise estimate of this function on $Q$.  This  lets us
reduce the weak-type estimate of this expression to the weak-type
estimate of $1_Q S (1_{Q}f)$ (with an error term that can be controlled).

\smallskip

By Theorem~\ref{thm:perez-twowt-max}, $\|M(f\sigma)\|_{L^p(u)}\leq
C\|f\|_{L^p(\sigma)}$.   Therefore, it remains to estimate the second
term in~\eqref{eqn:strong-reduction}.  We will show that each term in
the sum is bounded
by $C\tau\|f\|_{L^p(\sigma)}$.  Note that this gives us our estimate
which is cubic in $\tau$.

Again, following~\cite{HyLa}, we show that this reduces to a
two-weight estimate for a positive Haar shift operator.
We reorder the sum as follows:  fix an integer $i\in [1,\tau]$ and  sum over every cube
$Q=(Q_j^k)^i$ and then over all cubes $Q_r^s\in \mathcal{L}$
such that $(Q_r^s)^i = Q$. Then we have that
$$
\sli_{j,k} \av{|f|\sigma}{(Q_j^k)^i}\chi_{Q_j^k}
= \sli_{Q} \av{|f|\sigma}{Q}\sli_{\stackrel{R\in \mathcal{L}}{
    R^{i}=Q}} \chi_R = \sli_{Q} \chi_{Q}^i \av{|f|\sigma}{Q} = S^i_{\mathcal{L}}(|f|\sigma),
$$
where the last sum is taken over all dyadic cubes $Q$ and
$$
\chi_{Q}^i= \sli_{\stackrel{R\in \mathcal{L}}{ R^{i}=Q}} \chi_R.
$$

Clearly, $S_\mathcal{L}$ (hereafter we omit the superscript $i$) is a positive operator.  We claim that it is
in fact a positive Haar shift of complexity at most $(0,\tau-1)$.
From the definition we have that
$$
S_\mathcal{L}f=\sli_{Q\in \mathcal{L}} \frac{1}{|Q|}\chi_Q^i \ili_Q f,
$$
and so in the notation used above we have that
$$
S_Q = \frac{1}{|Q|} \chi_Q^i \ili_Q f,
$$
$Q' = Q$,  $h_{Q'}=\chi_Q$,  the $Q''$ are all the $(i-1)$-children of $Q$,
and $h_{Q''}= \sli_{R\in ch(Q'')} c_{R}\chi_R$, where $c_R=1$ if $R\in
\mathcal{L}$ and $c_R=0$ otherwise.  Thus $S_\mathcal{L}$ has a Haar shift kernel
of complexity $(0,i-1),\, i\le \tau$.  To see that it is bounded on $L^2$, we
use the properties of the cubes $Q_j^k$.  By duality, there exists
$g\in L^2$, $\|g\|_2=1$, such that
\begin{multline*}
\|S_\mathcal{L}f\|_2^2
 = \int_{\mathbb{R}^d}\sum_{j,k}
\av{f}{(Q_j^k)^i}\chi_{Q_j^k}(x)g(x)\,dx \\
 \leq 2 \sum_{j,k}\av{f}{(Q_j^k)^i} \av{g}{Q_j^k}|E_j^k|
 \leq 2 \int_{\mathbb{R}^d} Mf(x)Mg(x)\,dx.
\end{multline*}
The last integral is bounded by $\|f\|_2\|g\|_2$ by H\"older's inequality and
the unweighted $L^2$ inequality for the maximal operator.

\begin{zamech}  \label{rem:dual}
It follows from this argument that  both $S_\mathcal{L}$ and its
adjoint $S_{\mathcal{L}}^*$ are positive Haar shifts with uniform bounds.
\end{zamech}

\medskip

\noindent{\bf Definition}. Given a positive Haar shift operator $S$,  define the associated maximal
singular integral operator  by
$$
S_\sharp (x):=\sup_{0<\epsilon\le v<\infty} S_{\epsilon, v}(x) =
\sup_{0<\epsilon\le v<\infty}
\sum_{Q\in \mathcal{D},\, \epsilon \le \ell(Q)\le v} S_Q f(x)\,.
$$
To prove that $\|S_\mathcal{L}(f\sigma)\|_{L^p(u)}\leq C\tau
\|f\|_{L^p(\sigma)}$, we use the following result that is essentially
due to Sawyer and which can be found in Hyt\"onen and Lacey~\cite{HyLa} and
Hyt\"onen, {\em et al.}~\cite{HLM+}.  The precise statement
below is gotten by combining  Theorem 4.7 of \cite{HLM+} with Corollary 3.2 and
Lemma 3.3 of \cite{HyLa}.

\begin{theorem} \label{thm:testing}
Let $S$ be a positive Haar shift of complexity $(m,n)$. Then the
associated maximal singular integral $S_{\sharp}$ satisfies
\begin{multline} \label{eqn:testcond1}
\|S_{\sharp}(\cdot \sigma)\|_{L^p(\sigma)\to L^p(u)}  \\
\leqslant \tau \|M(\cdot \sigma)\|_{L^p(\sigma)\to L^p(u)} +
\sup_Q \frac{\|\chi_Q S(\chi_Q
  \sigma)\|_{L^p(u)}}{\sigma(Q)^{\frac{1}{p}}} +
\sup_Q \frac{\|\chi_Q S^*(\chi_Q u)\|_{L^p(\sigma)}}{u(Q)^{\frac{1}{p}}}.
\end{multline}
\end{theorem}

\begin{zamech}
  Testing conditions of this kind were first proved by
  E. Sawyer~\cite{S} for positive operators.  Later, this result was
  proved by Nazarov, Treil andVolberg~\cite{NTV,NTV-mrl} for all well localized operators (in
  particular, all Haar shifts) when $p=2$.  Also see
  Theorem~\ref{thm:testingCZM} below.  It is not known if
  Theorem~\ref{thm:testing} is true for all Haar shifts when $p\neq 2$ even if $S_\sharp$ is replaced by $S$.
\end{zamech}

\medskip

We will now use Theorem \ref{thm:testing} to show that if $S$ is any positive Haar
shift operator, then the right-hand side of~\eqref{eqn:testcond1} is
bounded with a constant linear in $\tau$, provided that our bump
conditions are satisfied.    As before, by
Theorem~\ref{thm:perez-twowt-max} the first term is bounded by
$C\tau$.  It remains to estimate the two testing conditions.  We will
estimate the first; the estimate for $S^*$ is gotten in essentially
the same fashion.  (See Remark~\ref{rem:second-bump} below.)

Fix a cube $Q_0$; using the notation from the definition of a Haar shift, we have that
\begin{equation}
\label{outin}
\chi_{Q_0} S(\chi_{Q_0} \sigma) = \sli_{R\subset Q_0} S_R(\sigma) +
\chi_{Q_0 }\sli _{R,\,Q_0\subset R}S_R(\chi_{Q_0} \sigma)\leqslant \sli_{R\subset Q_0}S_R(\sigma) + \chi_{Q_0} \av{\sigma}{Q_0}.
\end{equation}
The second inequality is straightforward:  see, for instance,
\cite{H,HyLa,HLM+,HPTV}.
As we noted above, the pair $(u,\sigma)$ satisfies the two-weight
$A_p$ condition~\eqref{eqn:alt-twowt-Ap}.   Therefore, the $L^p(u)$
norm of the second term is bounded by
\begin{equation*}
 \|\chi_{Q_0}\|_{L^p(u)}\av{\sigma}{Q_0} =
\av{u}{Q_0}^{1/p}\av{\sigma}{Q_0}^{1/p'} \sigma(Q_0)^{1/p} \leq C
\sigma(Q_0)^{1/p}\,.
\end{equation*}

To estimate the $L^p(u)$ norm of the first term, we form the following decomposition (see~\cite{HyLa}):
\begin{align*}
&\mathcal{K}=\mathcal{K}_i = \{Q\subset Q_0\colon \ell(I)=2^{i+\tau n}\},\, n\in \mathbb{Z}_+;\\
&\mathcal{K}_a = \{Q\in \mathcal {K}\colon 2^a \leqslant \av{u}{Q}^{\frac{1}{p}}\av{\sigma}{Q}^{\frac{1}{p'}} <2^{a+1}\}; \\
&\mathcal{P}^a_0 = \mbox{all maximal cubes in $\mathcal{K}_a$};\\
&\mathcal{P}^a_n = \left\{\mbox{maximal cubes $P'\subset P\in \mathcal{P}^a_{n-1}$, such that}\; \av{\sigma}{P'}>2\av{\sigma}{P} \right\}; \\
&\mathcal{P}^a = \bigcup_{n\geqslant 0}\mathcal{P}^a_n.
\end{align*}
Hereafter we suppress the index $i$; this will give us a sum with
$\tau+1$ terms.  Given $Q\in \mathcal{K}_a$, let $\Pi(Q)$ denote the minimal principal
cube that contains it, and define
\[ \mathcal{K}_a(P) = \{ Q \in \mathcal{K}_a : \Pi(Q)=P\}.  \]

We will estimate the $L^p(u)$ norm of the first sum on the right-hand
side of \eqref{outin} using the exponential decay distributional
inequality originated in \cite{LPR}.  (This inequality was subsequently improved in the sense of establishing the right dependence on the complexity of the shift in
\cite{HPTV, HyLa}.)  Below, $S$ is any positive generalized Haar shift that is
bounded on unweighted $L^2$.  In particular, we will take
$S$ to be one of the positive Haar shifts $S_{\mathcal{L}}$ from
above.

\begin{theorem}  \label{di}
There exists a constant $c$, depending only on the dimension and the
unweighted $L^2$ norm of the shift, such that
for any $P\in \mathcal{P}^a$,
$$
u\left(x\in P:\,|S_{\mathcal{K}_a(P)}(\sigma)|>t\frac{\sigma(P)}{|P|}\right)\lesssim e^{-c\,t}\,u(P).
$$
\end{theorem}

\vspace{.2in}

It follows from  Theorem \ref{di} that for some positive constant $c$,
\begin{equation}
\label{main}
\|\sli_{R\subsetneq Q} S_R(\sigma) \|_{L^p(u)}\leqslant
C\tau\sli_a \left(\sli_{P\in \mathcal{P}^a}u(P) \left(\frac{\sigma(P)}{|P|}\right)^{p}\right)^{\frac{1}{p}}.
\end{equation}

\bigskip

We sketch the proof of \eqref{main} following the  beautiful calculations in \cite{HyLa}:
$$
\sli_{R\subsetneq Q} S_R(\sigma)= \sli_{i=0}^\tau\sum_a \sum_{P \in \mathcal{P}^a} S_{\mathcal{K}^a(P) }(\sigma)\,,
$$
and so
$$
\|\sli_{R\subsetneq Q} S_R(\sigma) \|_{L^p(u)}\leqslant (\tau+1)\sli_a  \|\sum_{P\in \mathcal{P}^a} S_{\mathcal{K}^a(P)} (\sigma)\|_{L^p(u)}\,.
$$
Fix $a$.  Using Fubini's theorem we write
\begin{align*}
 & \|\sum_{P\in \mathcal{P}^a} S_{\mathcal{K}^a(P)} (\sigma)\|_{L^p(u)} \\
& \qquad \qquad = \bigg(\int\bigg(\sum_j\sum_{P\in \mathcal{P}^a} \chi_{\{ S_{\mathcal{K}^a(P)}
  (\sigma)\in (j, j+1)\frac{v(P)}{|P|}\}}S_{\mathcal{K}^a(P)}
(\sigma)(x)\bigg)^p\,u(x)\,dx\bigg)^{1/p} \\
& \qquad \qquad \le
 \sum_j (j+1)\bigg(\int\bigg[\sum_{P\in \mathcal{P}^a} \chi_{\{S_{\mathcal{K}^a(P)} (\sigma)\in (j, j+1)\frac{\sigma(P)}{|P|}\}}\frac{\sigma(P)}{|P|}\bigg]^p\, u(x)\,dx\bigg)^{1/p}\,.
\end{align*}
By the choice of the stopping cubes $P\in \mathcal{P}^a$ we have that
\[ \bigg[\sum_{P\in \mathcal{P}^a} \chi_{\{S_{\mathcal{K}^a(P)}
  (\sigma)\in (j,
  j+1)\frac{\sigma(P)}{|P|}\}}\frac{\sigma(P)}{|P|}\bigg]^p \lesssim
\sum_{P\in \mathcal{P}^a} \chi_{\{S_{\mathcal{K}^a(P)} (\sigma)\in (j,
  j+1)\frac{\sigma(P)}{|P|}\}}\bigg(\frac{\sigma(P)}{|P|}\bigg)^p\,.
\]
This follows because the ratios $\frac{\sigma(P)}{|P|}$  in the sum on
the left are super-exponential. This beautiful observation from
\cite{HyLa} lets us write
\begin{multline*}
 \|\sum_{P\in \mathcal{P}^a} S_{\mathcal{K}^a(P)}
 (\sigma)\|_{L^p(u)} \\
\lesssim
\sum_j (j+1)\bigg(\sum_{P\in \mathcal{P}^a}
\bigg(\frac{\sigma(P)}{|P|}\bigg)^p\,
u(S_{\mathcal{K}^a(P)} (\sigma)\in (j, j+1)\frac{\sigma(P)}{|P|})\,\bigg)^{1/p}\,.
\end{multline*}
Then by the  distributional  inequality from Theorem \ref{di}:
  $$
 \|\sum_{P\in \mathcal{P}^a} S_{\mathcal{K}^a(P)}
 (\sigma)\|_{L^p(u)}\lesssim
  \sum_j (j+1)e^{-cj/p}\bigg(\sum_{P\in \mathcal{P}^a} \bigg(\frac{\sigma(P)}{|P|}\bigg)^p\, u(P)\,\bigg)^{1/p}\,.
 $$
This gives us \eqref{main}.

\bigskip

It is at this point in the proof that we can no longer assume that our
pair of weights $(u,\sigma)$ satisfies the general $A_p$ bump condition and we must
instead make the more restrictive assumption that we have log bumps.   Before doing so, however, we want
to show how the proof goes and where the problem arises for general
bumps.    We will then give the modification necessary to make this
argument work for log bumps.

\smallskip

Define the sequence
\[
\mu_Q = \begin{cases} |P|, &Q=P \,, \mbox{for some cube $P\in \mathcal{P}^a$}\\
0, &\mbox{otherwise};\end{cases}
\]
then the inner sum in \eqref{main} becomes
$$
\sli_{Q\subset Q_0} \frac{u(Q)}{|Q|} \left(\frac{\sigma(Q)}{|Q|}\right)^{p}\mu_Q.
$$
But by H\"older's inequality in the scale of Orlicz spaces,
\begin{equation} \label{A}
\frac{\sigma(Q)}{|Q|} =
\av{\sigma^{\frac{1}{p}}\sigma^{\frac{1}{p'}}}{Q}
\leqslant C\|\sigma^{\frac{1}{p'}}\|_{Q, B}
\|\sigma^{\frac{1}{p}}\|_{Q, \bar{B}}
\leqslant \|\sigma^{\frac{1}{p'}}\|_{Q, B} \inf_{x\in Q} M_{\bar{B}}(\sigma^{\frac{1}{p}}\chi_Q).
\end{equation}
Therefore, by~\eqref{eqn:alt-separated-bumps},
\begin{equation}
\label{B}
\sli_{Q\subset Q_0} \frac{u(Q)}{|Q|} \left(\frac{\sigma(Q)}{|Q|}\right)^{p}\mu_Q
\leqslant K^p \sli_{Q\subset Q_0} \mu_Q \inf_{x\in Q} M_{\bar{B}}(\sigma^{\frac{1}{p}}\chi_Q)^p.
\end{equation}

To complete the proof we need two lemmas.  The first can be found in \cite{LPR}.

\begin{lemma}
$\{\mu_Q\}$ is a Carleson sequence.
\end{lemma}

The second is a folk theorem; a proof can be found in \cite{MP}.

\begin{lemma}
If $\{\mu_Q\}$ is a Carleson sequence, then
$$
\sli_{Q\subset Q_0} \mu_Q \inf_Q \chi_{Q_0}F(x) \lesssim \ili_{Q_0} F(x)dx.
$$
\end{lemma}

Combining these two lemmas with  Theorem~\ref{thm:perez-orlicz-max}
(since $\bar{B}\in B_p$) we
see that
\begin{multline*}
\sli_{Q} \frac{u(Q)}{|Q|} \left(\frac{\sigma(Q)}{|Q|}\right)^{p}\mu_Q
\leqslant K^p \sli_{Q\,,\, Q\subset  Q_0} \mu_Q \inf_{x\in Q}
M_{\bar{B}}(\sigma^{\frac{1}{p}}\chi_{Q_0})^p\\
\lesssim K^p \|M_{\bar{B}}(\sigma^{\frac{1}{p}}\chi_{Q_0})\|^p_{L^p(dx)}
\lesssim  K^p \|\sigma^{\frac{1}{p}}\chi_{Q_0}\|^p_{L^p(dx)} = K^p \sigma(Q_0).
\end{multline*}

\smallskip

This would complete the proof except that we must now sum over $a$,
and in \eqref{main} this sum goes from $-\infty$ to the logarithm of
the two-weight $A_p$ constant of the pair $(u,\sigma)$.    We cannot
evaluate this sum unless we can modify the above argument to yield a
decay constant in $a$.  In the one-weight argument in \cite{HyLa} the
authors could use the fact that the parameter $a$ run from $0$ to the
logarithm of $A_p$ constant:   this follows since by H\"older's
inequality the $A_p$ constant of any weight is at least $1$.    In
the two-weight case  the $A_p$ constant can be arbitrarily small, and
therefore we must sum over infinitely many values of $a$. We are able to
get the desired decay constant only by assuming that we are working
with log bumps.

We modify the above argument as follows.  Essentially, we will use the
properties of log bumps to replace $\bar{B}$ with a slightly larger
Young function.  Define $B_0(t)= t^{p'} \log (e
+t)^{p'-1+\frac{\delta}2}$; then we again have that $\bar{B}_0\in
B_{p}$.  Instead of \eqref{B} we will prove that there exists
$\gamma$, $0<\gamma<1$, such that
\begin{equation}
\label{ag}
\sli_{Q\subset Q_0} \frac{u(Q)}{|Q|} \left(\frac{\sigma(Q)}{|Q|}\right)^{p}\mu_Q
\leqslant K^{(1-\gamma)p} 2^{a\gamma p}\sli_{ Q\subset  Q_0} \mu_Q \inf_{x\in Q}
M_{\bar{B_0}}(\sigma^{\frac{1}{p}}\chi_{Q_0})^p.
\end{equation}
Given inequality \eqref{ag},  we can repeat the argument above, but we
now have the decay term $2^{a\gamma p}$ which allows us to sum in $a$
and get the desired estimate.

To prove \eqref{ag} suppose for the moment that there exists $\gamma$ such that
\begin{equation}
\label{interp1}
 \|\sigma^{\frac{1}{p'}}\|_{Q, B_0}
\le C_1  \|\sigma^{\frac{1}{p'}}\|_{Q, B}^{1-\gamma}  \|\sigma^{\frac{1}{p'}}\|_{L^{p'}(Q, dx/|Q|)}^{\gamma}\,.
 \end{equation}
Given this, fix a cube $Q\in
\mathcal{P}^a$---we can do this since otherwise $\mu_Q=0$.  Then
\begin{align*}
&  \frac{u(Q)}{|Q|} \left(\frac{\sigma(Q)}{|Q|}\right)^{p}  \\
& \qquad \qquad \le \langle u\rangle_Q \| \sigma^{1/p'}\|_{B_0, Q} ^p\|
\sigma^{1/p}\|_{\bar{B}_0, Q}^p \\
& \qquad \qquad \le
\langle u\rangle_Q \| \sigma^{1/p'}\|_{B,
  Q}^{(1-\gamma)p}\|\sigma^{1/p'}\|_{L^{p'}(Q, dx/|Q|)}^{\gamma p}
\|\sigma^{1/p}\|_{\bar{B}_0, Q}^p \\
& \qquad \qquad = (\langle u\rangle_Q^{1/p}\| \sigma^{1/p'}\|_{B,
  Q})^{(1-\gamma)p}
\cdot(\langle u\rangle_Q^{1/p}\|\sigma^{1/p'}\|_{L^{p'}(Q, dx/|Q|)})^{\gamma p}
\cdot \| \sigma^{1/p}\|_{\bar{B}_0, Q}^p \\
& \qquad \qquad  \le K^{(1-\gamma)p}
\cdot (\langle
u\rangle_Q^{1/p}\langle\sigma\rangle_Q^{1/p'})^{\gamma\,p}\cdot \|
\sigma^{1/p}\|_{\bar{B}_0, Q}^p\\
& \qquad \qquad \le K^{(1-\gamma)p}\cdot  2^{a\gamma p}\cdot
\|\sigma^{1/p}\|_{\bar{B}_0, Q}^p \\
& \qquad \qquad \le K^{(1-\gamma)p}\cdot  2^{a\gamma p}\cdot
\inf_{x\in Q}
M_{\bar{B_0}}(\sigma^{\frac{1}{p}}\chi_{Q_0})^p.
\end{align*}
Inequality~\eqref{ag} now follows immediately.

Therefore,  to complete
the proof we must establish \eqref{interp1}.  By the rescaling
properties of the Luxemburg norm \cite[Section~5.1]{CU-M-P-book}, the
right-hand side of this inequality is equal to
\[  \|\sigma^{\frac{1-\gamma}{p'}}\|_{C,Q}
\|\sigma^{\frac{\gamma}{p'}}\|_{p'/\gamma,Q}, \]
where $C(t)=B(t^{\frac{1}{1-\gamma}})$.   Therefore, by the generalized
H\"older's inequality in Orlicz spaces
(\cite[Lemma~5.2]{CU-M-P-book}), inequality \eqref{interp1} holds if
for all $t>1$,
\begin{equation} \label{balance}
C^{-1}(t) t^{\frac{\gamma}{p'}} \lesssim B_0^{-1}(t).
\end{equation}
A straightforward calculation (see \cite[Section~5.4]{CU-M-P-book}) shows
that
\[ C^{-1}(t) = B^{-1}(t)^{1-\gamma} \approx
\frac{t^{\frac{1-\gamma}{p'}}}{\log(e+t)^{\frac{1-\gamma}{p}+\frac{\delta(1-\gamma)}{p'}}},
\qquad
B_0^{-1}(t)\approx
\frac{t^{\frac{1}{p'}}}{\log(e+t)^{\frac{1}{p}+\frac{\delta}{2p'}}}. \]
By equating the exponents on the logarithm terms, we see that
\eqref{balance} holds if we take
\[ \gamma = \frac{\delta}{2( p'-1+\delta)}. \]
Therefore, with this value of $\gamma$ inequality \eqref{interp1} holds, and this completes
our proof.

\smallskip

For the convenience of the reader we give a direct proof of
\eqref{interp1};  this computation will also be used below in
Section~\ref{ll}.   The desire inequality obviously follows from the
following lemma.  

\begin{lemma}
 \label{interp2}
 Given a probability measure $\mu$, let $f$ be a non-negative
 measurable function.  Let $B, B_0$ be logarithmic bumps as in
 \eqref{eqn:right-bump}  with $\delta=\tau$ and $\delta=\frac{\tau}2$ respectively.
 Then there exists an absolute constant $C$ and $\gamma=\gamma (p',\tau)>0$ such that
 \begin{equation}
 \label{interp3}
 \|f\|_{B_0, \mu} \le C\,\|f\|_{B, \mu}^{1-\gamma}\, \|f\|_{L^{p'}(\mu)}^{\gamma}\,.
 \end{equation}
  \end{lemma}
\begin{proof}
  We will actually show that $\gamma =\frac1{2+(p'-1)\frac2{\tau}}$.
  Define $\Delta:= \int |f|^{p'}\,d\mu$.    Since
  inequality~\eqref{interp3} is homogeneous, we may assume without
  loss of generality that 
\begin{equation}
\label{1}
\|f\|_{B, \mu}=1\,.
\end{equation}
Moreover, we may assume that $\Delta\leq 1$:
otherwise \eqref{interp3} can be achieved by choosing $C$ sufficiently
large.
Let $\epsilon<1$ and $K$  be constants; we will determine their
precise value (in this order) below.   Then we have that
\begin{align*}
& \int \frac{f^{p'}}{\epsilon^{p'}}\log\bigg(e
+\frac{f}{\epsilon}\bigg)^{p'-1+\frac{\tau}2}\,d\mu \\
& \qquad \le \int_{\{f \le K\epsilon\}}\dots + \int_{\{f\ge K\epsilon\}}\dots \\
& \qquad \le \frac{\Delta}{\epsilon^{p'}} [\log (e+K)]^{p'-1+\frac{\tau}2}
+\int_{\{f\ge K\epsilon\}}
\frac{f^{p'}}{\epsilon^{p'}}\frac{\log(e
  +\frac{f}{\epsilon})^{p'-1+\tau}}{[\log (e+K)]^{\frac{\tau}2} }\,d\mu  \\
& \qquad \le \frac{\Delta}{\epsilon^{p'}} [\log (e+K)]^{p'-1+\frac{\tau}2} +
\int
\frac{f^{p'}}{\epsilon^{p'}}\frac{\log(\frac{e}{\epsilon}
  +\frac{f}{\epsilon})^{p'-1+\tau}}{[\log (e+K)]^{\frac{\tau}2} }\,d\mu \\ 
& \qquad \le \frac{\Delta}{\epsilon^{p'}} [\log
(e+K)]^{p'-1+\frac{\tau}2} \\
& \qquad \qquad +
\frac1{\epsilon^{p'}[\log (e+K)]^{\tau/2}} \bigg[ \int
f^{p'}\log^{p'-1+\tau}(e +f)\,d\mu  + 
\int f^{p'}\log (\epsilon^{-1})^{p'-1+\tau}\,d\mu\bigg] \\
& \qquad \le  \frac{\Delta}{\epsilon^{p'}} [\log (e+K)]^{p'-1+\frac{\tau}2} +
\frac1{\epsilon^{p'}[\log (e+K)]^{\tau/2}} \bigg[ 1+  \Delta \log (\epsilon^{-1})^{p'-1+\tau}\bigg] \,.
\end{align*}
In the last line we used \eqref{1}. Fix $\epsilon$ so that 
$$
\Delta = (\epsilon^{p'})^{1+c}\,,
$$
where $c=1+(p'-1)\frac2{\tau}$.   In other words,
$$
\epsilon = (\Delta^{1/p'})^{\gamma}= \|f\|_{L^{p'}(\mu)}^{\gamma}\,,\,\gamma=\frac1{1+c}\,.
$$
Now choose  $K$ so that 
$$
[\log (e+K)]^{\tau/2} \approx \epsilon^{-p'}\,;
$$
then
$$
[\log (e+K)]^{p'-1+ \tau/2} \approx  (\epsilon^{-p'})^{1+(p'-1)\frac2{\tau}}
=:(\epsilon^{-p'})^{c}\,.
$$
If we substitute these values into the above calculation, we see that
the right hand side is dominated by a constant.  Hence,  by the
definition of the  Luxemburg norm, 
$$
\|f\|_{B_0,\mu}\le C\, \epsilon= C\,\|f\|_{L^{p'}(\mu)}^{\gamma} \,.
$$
This completes the proof.
\end{proof}

\begin{zamech} \label{rem:second-bump}
The conjugate testing condition
  can be verified similarly.  By Remark~\ref{rem:dual} the adjoint
  $S^*$ is also a Haar shift, and so we can apply the distribution inequality from Theorem \ref{main}
  to  it.   Also, the second sum in
  \eqref{outin} will have the same pointwise estimate (exchanging
  $\sigma$ and $v$) if we replace $S$ with $S^*$.
\end{zamech}

\begin{zamech}
  In the proof of the first testing condition we only used the bump
  condition~\eqref{eqn:alt-separated-bumps}; to prove the second
  testing condition we use the second bump
  condition~\eqref{eqn:alt-separated-bumps-weak}.
\end{zamech}

\section{Proof of Theorem~\ref{thm:dyadic-result-weak}}
\label{section:proof-weak}

The proof of the weak-type inequality uses essentially the same
argument as above; here we sketch the changes required.  We repeat the
argument that yields inequality~\eqref{eqn:strong-reduction},
replacing the $L^p(u)$ norm with the $L^{p,\infty}(u)$ norm.  Since
the pair $(u,\sigma)$ satisfies the two-weight $A_p$ condition we have the
well known inequality that
\[ \|M(f\sigma)\|_{L^{p,\infty}(u)} \leq C\|f\|_{L^p(\sigma)}, \]
where the constant $C$ depends only on the $A_p$ constant and the
dimension.  Therefore it remains to estimate the $L^{p,\infty}(u)$
norm of $S_{\mathcal{L}}(|f|\sigma)$.  However,
  from Hyt\"onen, {\em et al.}~\cite[Theorem~4.3]{HLM+} we have the following analog of
  Theorem~\ref{thm:testing}.

 \begin{theorem} \label{thm:testing-weak}
Let $S$ be a positive Haar shift of complexity $(m,n)$. Then
\[
\|S(\cdot \sigma)\|_{L^p(\sigma)\to L^{p,\infty}(u)}
\leqslant \tau \|M(\cdot \sigma)\|_{L^p(\sigma)\to L^{p,\infty}(u)} +
\sup_Q \frac{\|\chi_Q S^*(\chi_Q u)\|_{L^p(\sigma)}}{u(Q)^{\frac{1}{p}}}.
\]
\end{theorem}

Given Theorem~\ref{thm:testing-weak} the argument now proceeds exactly
as before, using the bump
condition~\eqref{eqn:alt-separated-bumps-weak} to bound the testing
condition.  This completes the proof.

\section{Proof of Theorems~\ref{thmllst} and~\ref{thmllw}}
\label{ll}

Our proofs are very similar to the proofs given in
Sections~\ref{section:proof} and~\ref{section:proof-weak}, so we will
describe the principle changes.  Following the argument in
Section~\ref{section:prelim}, it will suffice to prove the corresponding
results for dyadic shifts.

\begin{theorem} \label{thmllst-dyadic} Given $p$, $1<p<\infty$,
  suppose $A$ and $B$ are loglog-bumps of the form \eqref{llA},
  \eqref{llB} with $\delta>0$ {\bf sufficiently large}, and the pair
  of weights $(u,\sigma)$ satisfies~\eqref{eqn:separated-bumps}
  and~\eqref{eqn:separated-bumps-weak}.  Given any dyadic shift $S$ of
  complexity $(m,n)$, $\tau=\max(m,n)+1$, $\|S(f\sigma)\|_{L^p(u)}\leq
  C\tau^3\|f\|_{L^p(\sigma)}$, where $C$ depends only on the dimension
  $d$ and the suprema in~\eqref{eqn:separated-bumps}
  and~\eqref{eqn:separated-bumps-weak}.
\end{theorem}

\begin{theorem} \label{thmllw-dyadic}
Given $p$, $1<p<\infty$, suppose  $A$ is a loglog-bump of the
form \eqref{llA} with $\delta>0$ {\bf sufficiently large},
and  the
pair of weights $(u,\sigma)$ satisfies~\eqref{eqn:separated-bumps-weak}.  Given any dyadic shift $S$ of complexity $(m,n)$,
$\|S(f\sigma)\|_{L^{p,\infty}(u)}\leq C\tau^2\|f\|_{L^p(\sigma)}$, where $C$ depends only on
the dimension $d$ and the supremum in~\eqref{eqn:separated-bumps-weak}.
\end{theorem}

We will  prove Theorem~\ref{thmllst-dyadic} by modifying the proof of Theorem
\ref{thm:dyadic-result} above;  Theorem~\ref{thmllw-dyadic} is proved
similarly.  The main step is to adapt Lemma \ref{interp2} to work with loglog-bumps.
Let $B$ be as in \eqref{llB}, and define $B_0$ similarly but with
$\delta$ replaced by $\delta/2$. Then arguing almost exactly as we did
in the proof of Lemma \ref{interp2},  we have that
\begin{equation}
\label{eps}
\|f\|_{B_0, \mu} \le C\|f\|_{B, \mu}\; \varepsilon\bigg(\frac{\|f\|_{L^{p'}(\mu)}}{\|f\|_{B,\mu}}\bigg)\,,
\end{equation}
where $\varepsilon (t) = (\log\frac{C}{t})^{-\kappa}$,
$C=C(p,\delta)$, and $ \kappa=\kappa(p, \delta)$ with $p\kappa >1$ if
$\delta$ is large enough. 

\begin{zamech}
To get the inequality $p\kappa>1$, it follows from the proof that it suffices to take $\delta>(p-1)^{-1}$.
\end{zamech}

Given~\eqref{eps} we have that 
\begin{align*}
&  \frac{u(Q)}{|Q|} \left(\frac{\sigma(Q)}{|Q|}\right)^{p}  \\
& \qquad \qquad \le  C\langle u\rangle_Q\| \sigma^{1/p'}\|_{B_0, Q} ^p\|
\sigma^{1/p}\|_{\bar{B}_0, Q}^p \\
& \qquad \qquad \le C\langle u\rangle_Q \| \sigma^{1/p'}\|_{B, Q} ^p\;
 \varepsilon \bigg(\frac{\langle
   \sigma\rangle_Q^{1/p'}}{\|\sigma^{1/p'}\|_{B, Q} }\bigg )^p
\| \sigma^{1/p}\|_{\bar{B}_0, Q}^p \\
& \qquad \qquad \le C\, \big(\langle u\rangle_Q^{1/p}\| \sigma^{1/p'}\|_{B, Q}\big)^p
 \varepsilon \bigg(\frac{\langle u\rangle_Q^{1/p}\langle
   \sigma\rangle_Q^{1/p'}}
{\langle u\rangle_Q^{1/p}\|\sigma^{1/p'}\|_{B, Q}  } \bigg)^p \|\sigma^{1/p}\|_{\bar{B}_0, Q}^p\,.
\end{align*}

To complete the proof, we need  a bound in $a$ for the first two
terms.  Since $Q\in \mathcal{P}^a$, we have that  $\langle
u\rangle_Q^{1/p}\langle \sigma\rangle_Q^{1/p'}\asymp 2^a$.  Let
$b_0-1$ be the logarithm of the supremum in~\eqref{eqn:separated-bumps}.  Then there
exists $b$, $a\leq b \leq b_0$, such that 
$\langle u\rangle_Q^{1/p}\|\sigma^{1/p'}\|_{B, Q} \asymp 2^b$.   Our sum in
$a$ will go from $-\infty$ to $b_0$, so it will suffice to consider
those terms where $a<0$.  

If $b$ is negative and $|b|\ge |a|/2$,  then the first term is bounded
by $2^{-\frac{p|a|}2}$, and the second is bounded by some constant.
If $b>0$ or if it is negative but $|b|\le \frac{|a|}2$, then we
estimate the first term by $2^{b_0}$, and the argument in the second
term is at most $ 2^{-\frac{|a|}2}$. Hence the second term is bounded by $\frac{C}{|a|^{p\kappa}}$.
By our assumption, $p\kappa>1$ and so the series $\sum_{a<0} [2^{\frac{pa}2} + \frac{C}{|a|^{p\kappa}}]$
converges.  Therefore, we can finish the proof of Theorem~\ref{thmllst}
exactly as in the proof in Section \ref{section:proof}.

\section{Counterexamples to Muckenhoupt--Wheeden conjectures}
\label{ce}

In this section we prove that the weak-type conjecture of Muckenhoupt
and Wheeden discussed in the Introduction is false for the Hilbert
transform when $p=2$.  We in fact prove a stronger result.

For brevity, we introduce some additional notation.  Let
$\sigma=v^{-1}$ and let $M_\sigma f= M(f\sigma)$.  Define $M_u$,
$H_\sigma$ and $H_u$ similarly, where $H$ is the Hilbert transform.
Then we can reformulate the conjecture as follows: if
\begin{equation}
\label{Mi}
 M_u : L^{2}(u)
\rightarrow L^{2}(\sigma).
\end{equation}
then
\begin{equation} \label{weaki}
H_\sigma : L^2(\sigma)\rightarrow L^{2, \infty}(u).
\end{equation}

We will show by contradiction that this is not true in general.
Suppose to the contrary that if the pair $(u,v)$ satisfies \eqref{Mi},
then \eqref{weaki} holds.  Then for any $f\in L^2(u)$ and any cube $Q$,
\[
\int_Q H_\sigma f u \, dx \leq \|H_\sigma
f\|_{L^{2,\infty}(u)}\|\chi_Q\|_{L^{2,1}(u)}
\leq \|H_\sigma\|_{L^2(\sigma)\rightarrow L^{2,\infty}(u)}
\|f\|_{L^2(\sigma)}u(Q)^{1/2}.
\]
Let $f=H_u(\chi_Q)$.  Then by duality (since $H_\sigma$ is the adjoint
of $H_u$) we have that the pair $(u,v)$ satisfies the testing
condition
\begin{equation} \label{eqn:testing1}
\int_Q |H_u(\chi_Q)|^2 \sigma \,dx \leq Cu(Q).
\end{equation}

The same argument shows that if the pair $(u,\sigma)$ satisfies
\begin{equation}
\label{Mi2}
 M_\sigma : L^{2}(\sigma)
\rightarrow L^{2}(u),
\end{equation}
then this pair also satisfies the testing condition
\begin{equation} \label{eqn:testing2}
\int_Q |H_\sigma(\chi_Q)|^2 u\,dx \leq C\sigma(Q).
\end{equation}

However, we have the following testing condition result for the
Hilbert transform when $p=2$.  This was proved in
\cite[Chapter 22]{Vo} (see also \cite{NTVlost}).

\begin{theorem} \label{thm:testingCZM}
Let $H$ be the Hilbert transform. Then
\begin{multline*}
\|H(\cdot \sigma)\|_{L^2(\sigma)\to L^2(u)}
\leqslant  \|M(\cdot \sigma)\|_{L^2(\sigma)\to L^2(u)} +  \|M(\cdot u)\|_{L^2(u)\to L^2(\sigma)}+\\
\sup_Q \frac{\|H(\chi_Q
  \sigma)\|_{L^2(u)}}{\sigma(Q)^{\frac{1}{2}}} +
\sup_Q \frac{\|H(\chi_Q u)\|_{L^2(\sigma)}}{u(Q)^{\frac{1}{2}}}.
\end{multline*}
\end{theorem}

\begin{zamech}
  We note that the proofs in \cite{Vo} and \cite{NTVlost} can be
  adapted to prove Theorem~\ref{thm:testingCZM} for an {\it arbitrary}
  operator with Calder\'on--Zygmund kernel.  This was essentially done
  in \cite{PTV}.  This paper is concerned with one-weight
  inequalities, but the same argument works with no change in the two
  weight case.
\end{zamech}

Therefore, by our assumption and Theorem~\ref{thm:testingCZM}, we have
that if a pair of weights $(u,\sigma)$ satisfies~\eqref{Mi}
and~\eqref{Mi2}, then $H : L^2(\sigma)\rightarrow L^2(u)$.  However,
this contradicts the counterexample constructed by Reguera and
Scurry~\cite{RS} (in fact, in this paper they directly disprove the
testing condition for the Hilbert transform).  Therefore, the
weak-type conjecture of Muckenhoupt and Wheeden cannot hold.

In fact, we have proved a
stronger result.

\begin{theorem} \label{thm:weak-MW}
There exists a pair of weights $(u,\sigma)$ such that \eqref{Mi} and
\eqref{Mi2} hold, but the Hilbert transform does not satisfy the
weak-type inequality $H_\sigma : L^2(\sigma) \rightarrow
L^{2,\infty}(u)$.
\end{theorem}

\bigskip

We conclude this section with three remarks.  First, Theorem
\ref{thm:testingCZM} is essentially Theorem
22.3 in \cite{Vo}, but it is formulated there in slightly different
language. For the convenience of the reader we want to explain why
these two results are in fact equivalent.

The main part of Theorem 22.3 in \cite{Vo} says that if $M_u,
M_\sigma, H_u, H_\sigma$ all satisfy the testing conditions
\eqref{eqn:testing1} and \eqref{eqn:testing2} (replacing $H$ with $M$ when
dealing with the maximal operator), then all four of them are bounded
in corresponding pairs of weighted spaces. This gives us
Theorem \ref{thm:testingCZM}. Conversely, suppose that the right-hand
side of the inequality in Theorem \ref{thm:testingCZM} is finite. Then
\eqref{eqn:testing1} and \eqref{eqn:testing2} are both satisfied. Moreover,
$M_u$ and $M_\sigma$ are both bounded on the corresponding pairs of
weighted spaces. Therefore, trivially, $M_u, M_\sigma$ also satisfy
the corresponding testing conditions. Thus, all four operators satisfy
the testing conditions, and then Theorem~22.3 gives that
$H_u: L^2(u)\rightarrow L^2(\sigma)$ and $H_\sigma
:L^2(\sigma)\rightarrow L^2(u)$ as well as the corresponding norm
inequalities for the maximal operators. (The latter follows from Sawyer's
testing criterion for  the two weight boundedness of maximal function
\cite{Smax}.)

\bigskip

In the Introduction we noted that the weak-type conjecture we just
disproved followed from another  conjecture of Muckenhoupt and
Wheeden:  that
\begin{equation}
\label{Hweak}
u(\{x: |Hf(x)| >t\}) \le \frac{C}{t}\int |f|\, Mu\, dx\,.
\end{equation}
(This implication is a straightforward duality argument: see~\cite{CU-M-P-book}.)
Inequality \eqref{Hweak} was disproved by Reguera and Thiele \cite{RT}; our result
above gives another (indirect) proof of this fact.

\medskip

Finally, we note that there is a weaker conjecture than \eqref{Hweak}
which has also been shown to be false.   In the one-weight case it was
conjectured that for all $w\in A_1$,
\begin{equation}
\label{HweakA1}
w(\{x: Hf(x) >t\}) \le \frac{C}{t}[w]_{A_1}\int |f|\, w\, dx\,,
\end{equation}
where $[w]_{A_1} = \|\frac{Mw}{w}\|_{L^\infty}$.  The counterexample
to \eqref{Hweak} in \cite{RT} is not in $A_1$. Essentially, disproving \eqref{HweakA1} amounts to
finding a ``smooth" bad weight, which is even more difficult to build
than the weight of Reguera--Thiele.  While no explicit example has
been constructed, the existence of such a weight has been proved using
Bellman function techniques: in \cite{NRVV} it was shown that there exist
weights in $A_1$ such that
\begin{equation}
\label{HweakA1be}
\|H\|_{L^1(w)\rightarrow L^{1,\infty}(w)} \ge c\,[w]_{A_1} \log^{1/5} [w]_{A_1} \,.
\end{equation}


\begin{thebibliography}{99}
\label{rf}


\bibitem{ChWW} {\sc A. Chang, M. Wilson, T. Wolff}, {\em Some weighted norm inequalities concerning the Schr\"odinger operators},  Comment. Math. Helvetici, {\bf 60} (1985), 217--286.


\bibitem{CU-Ma-Pe} {\sc D. Cruz-Uribe, J. M. Martell, C. P\'erez,}
  {\it Sharp weighted estimates for classical operators},  Adv. in
  Math., 229 (2012), 408--441.

\bibitem{CU-M-P-book}  {\sc D. Cruz-Uribe, J. M. Martell,
    C. P\'erez,} {\em Weights, Extrapolation and the Theory of Rubio
    de Francia}, Operator Theory: Advances and Applications, 215,
  Birkhauser, Basel, (2011).


\bibitem{CU-Ma-Pe07} {\sc D. Cruz-Uribe, J. M. Martell, C. P\'erez,} {\em Sharp two-weight inequalities for singular integrals, with applications to the Hilbert transform and the Sarason conjecture}. Adv. Math. 216 (2007), no. 2, 647--676.

\bibitem{cruz-uribe-perez99}
{\sc D.~Cruz-Uribe and C.~P{\'e}rez},
{\em Sharp two-weight, weak-type norm inequalities for singular integral
  operators},
Math. Res. Lett., 6(3-4):417--427, 1999.

\bibitem{cruz-uribe-perez00b}
{\sc D.~Cruz-Uribe and C.~P{\'e}rez}
{\em Two-weight, weak-type norm inequalities for fractional integrals,
  Calder\'on--Zygmund operators and commutators},
{ Indiana Univ. Math. J.}, 49(2):697--721, 2000.



\bibitem{CU-P-pisa} {\sc D. Cruz-Uribe,  C. P\'erez,}
{\em On the two-weight problem for singular integral operators},
{Ann. Sc. Norm. Super. Pisa Cl. Sci. (5)}, 1(4) (2002), 821--849.

\bibitem{CF}{\sc C. Fefferman},  {\em The uncertainty principle},  Bull. Amer. Math. Soc., {\bf 9} (1983), 129--206.

\bibitem{H} {\sc T. Hyt\"{o}nen}, {\em The sharp weighted bound for general Calder\'on-Zygmund operators},
Ann. of Math., to appear.

\bibitem{H2} {\sc T. Hyt\"{o}nen}, {\em Representation of singular
    integrals by dyadic operators, and the $A_2$  theorem},
  ArXiv:1108.5119v1.

\bibitem{HyLa} {\sc  T. Hyt\"onen, M. T. Lacey},  {\em The
    $A_p-A_\infty$ inequality for general Calder\'on-Zygmund operators}, ArXiv:1106.4797v1.


\bibitem{HLM+} {\sc T. Hyt\"onen, M. T. Lacey, H. Martikainen, T. Orponen, M. Reguera, E. Sawyer, and I. Uriarte-Tuero}, {\it Weak and strong type estimates for maximal truncations of Calderón-Zygmund operators on $A_p$ weighted spaces}, arXiv:1103.5229.

\bibitem{HyPer} {\sc T. Hyt\"onen, C. P\'erez,} {\it Sharp weighted bounds involving $A_\infty$}, arXiv:1103.5562.

\bibitem{HPTV} {\sc T. Hyt\"{o}nen, C. P\'erez, S. Treil, A. Volberg}, {\em  Sharp weighted estimates for dyadic shifts and the $A_2$ conjecture}, arXiv:1010.0755v2;

\bibitem{Le} {\sc A. Lerner} {\it A pointwise estimate for local sharp maximal function with applications to singular integrals.} Bull. London Math. Soc., 42:843--856, 2010.

\bibitem{LPR} {\sc M. Lacey, S. Petermichl, M. Reguera}, {\it Sharp
    $A_2$ inequality for Haar shift operators.}  Math. Ann. 348
  (2010), no. 1, 127Ð141, see also arXiv:0906.1941.

  \bibitem{MP} {\sc J. Moraes and M. C. Pereyra}, {\em Weighted
    estimates for dyadic paraproducts and $t$-Haar multipliers with
    complexity $(m,n)$}, ArXiv:1108.3109.

\bibitem{NRVV} {\sc F. Nazarov, A. Reznikov, V. Vasyunin, A. Volberg}, {\em $A_1$ conjecture: weak norm estimates of weighted singular operators and Bellman functions}, Preprint, 1--10, 2011, sashavolberg.wordpress.com.

\bibitem{NRV1} {\sc F. Nazarov, A. Reznikov, A. Volberg}, {\em The sharp bump condition for the two-weight problem for classical
singular integral operator: the Bellman function approach}, preprint, 1--3, 2011.

\bibitem{NTV} {\sc F. Nazarov, S. Treil, A. Volberg,} {\it Bellman
    function and two-weight inequality for martingale transform.},
  J. of Amer. Math.Soc., {\bf 12}, (1999), no. 4, 909--928

  \bibitem{NTV97} {\sc F. Nazarov, S. Treil, A. Volberg}, {\em Cauchy integral and Calder\'on--Zygmund operators on nonhomogeneous spaces}, Int. Math. Res. Not., (1997), No. 15,  703--726.

\bibitem{NTV98} {\sc F. Nazarov, S. Treil, A. Volberg}, {\em Weak type
    estimates and Cotlar inequalities for Calder\'on-Zygmund operators
    on nonhomogeneous spaces},  Int. Math. Res. Not., 1998, no. 9, 463--487.

\bibitem{NTV-acta} {\sc F. Nazarov, S. Treil, A. Volberg}, The
  $Tb$-theorem on non-homogeneous spaces, Acta Math., 190 (2003),
  151--239;

\bibitem{NTV-mrl} {\sc F. Nazarov, S. Treil, A. Volberg,} {\em Two weight
  inequalities for individual Haar multipliers and other well
  localized operators}. Math. Res. Lett. 15 (2008), no. 3, 583--597.

  \bibitem{NTVlost} {\sc F. Nazarov, S. Treil, A. Volberg,}  {\em
Two weight estimate for the Hilbert transform and corona decomposition for non-doubling measures}. arXiv:1003.1596.

\bibitem{NV} {\sc F. Nazarov, A. Volberg,} {\it A simple sharp
    weighted estimate of the dyadic shifts on metric spaces with
    geometric doubling},  arXiv:1104.4893

    \bibitem{NRTV} {\sc F. Nazarov, A. Reznikov, S. Treil, A. Volberg}, {\em The sharp bump condition for the two-weight problem for classical
singular integral operator: the Bellman function approach}, preprint, Oct. 2011, 1--4.

 \bibitem{neugebauer83}
{\sc C.~J. Neugebauer},
{\em Inserting {$A\sb{p}$}-weights},
{ Proc. Amer. Math. Soc.}, 87(4):644--648, 1983.





  \bibitem{Pe94}  {\sc C. P\'erez} {\it Two weighted inequalities for potential and fractional  type Maximal operators}, Indiana Univ. Math. J., {\bf 43}, No. 2 (1994),  663--684.

  \bibitem{Pe94JL} {\sc C. P\'erez} {\it Weighted norm inequalities
      for singular integral operators}, J. London Math. Soc., {\bf 49}
    (1994), No. 2, 296--308.

\bibitem{Pe} {\sc C. P\'erez} {\it On sufficient conditions for the
    boundedness of the Hardy-Littlewood maximal operator between
    weighted $L^p$-spaces with different weights.} Proc. London
  Math. Soc.  {\bf 71} (1995),  No. 3,  135--157.

\bibitem{PeWh} {\sc C. P\'erez, R. Wheeden}, {\em Uncertainty
    principle estimates for vector fields}, J. of Funct. Analysis,
  {\bf 181} (2001), 146--188.


\bibitem{PTV} {\sc C. P\'erez, S. Treil, A. Volberg}, {\em On $A_2$ conjecture and corona decomposition of weights}, arXiv:1006.2630


\bibitem{RS} {\sc M. Reguera and J. Scurry}, {\em On joint estimates
    for maximal functions and singular integrals in weighted spaces},
  arXiv:1109.2027.

\bibitem{RT} {\sc M. Reguera and C. Thiele}, {\em The Hilbert
    transform does not map $L^1(Mw)$ to $L^{1,\infty}(w)$}, arXiv:1011.1767.

\bibitem{Smax} {\sc E. Sawyer} {\it A characterization of a two-weight
    norm inequality for maximal operators}. Studia Math., 75(1):1--11,
  1982.

\bibitem{S} {\sc E. Sawyer} {\it A characterization of two weight norm
    inequalities for fractional and Poisson integrals},
  Trans. Amer. Math. Soc. 308 (1988), 533--545.

\bibitem{TVZ} {\sc S.~Treil, A.~Volberg, D.~Zheng}, {\em Hilbert transform, Toeplitz operators and
Hankel operators, and invariant $A_{\infty}$ weights}, Revista Mat.
Iberoamericana, {\bf 13}, (1997), no. 2, 319--360.

\bibitem{Vo} {\sc A. Volberg}, {\em Calder\'on--Zygmund capacities and operators on nonhomogeneous spaces},
CBMS Regional Conference Series in Mathematics, Amer. Math. Soc., v. 100 (2003), 1--167.
\end{thebibliography}
\end{document}